\documentclass[review]{elsarticle}

\usepackage{bm,amssymb,amsmath,mathrsfs}
\usepackage{amsthm}
\usepackage{lineno}
\usepackage[usenames]{color}

\usepackage{dcolumn}

\newcommand{\bq}{\begin{equation}}
\newcommand{\eq}{\end{equation}}
\newcommand{\bc}{\begin{center}}
\newcommand{\ec}{\end{center}}
\newcommand{\bit}{\begin{itemize}}
\newcommand{\eit}{\end{itemize}}
\newcommand{\ben}{\begin{enumerate}}
\newcommand{\een}{\end{enumerate}}

\theoremstyle{plain}
\newtheorem{theorem}{Theorem}[section]
\newtheorem*{theorem*}{Theorem}
\newtheorem{proposition}[theorem]{Proposition}

\newtheorem{corollary}[theorem]{Corollary}

\newtheorem{definition}[theorem]{Definition}
\newtheorem{conjecture}[theorem]{Conjecture}

\begin{document}

\journal{(internal report CC23-3)}

\begin{frontmatter}

\title{Asymptotic results for the Poisson distribution of order $k$}

\author[cc]{S.~R.~Mane}
\ead{srmane001@gmail.com}
\address[cc]{Convergent Computing Inc., P.~O.~Box 561, Shoreham, NY 11786, USA}

\begin{abstract}
The Poisson distribution of order $k$ is a special case of a compound Poisson distribution.
Its mean and variance are known, but results for its median and mode are difficult to obtain,
although a few cases have been solved and upper/lower bounds for the mode have been established.
This note points out that {\em asymptotic results}, for both the median and mode, exhibit simple patterns.
The calculations are numerical, hence the results are presented as conjectures.
The purpose of this note is to discern patterns for the median and mode,
including expressions for exact limits and rates of convergence and (possibly sharp) upper/lower bounds, in a sense to be made precise in the text.
The derivation of proofs of the results is left for future work.
\end{abstract}

\vskip 0.25in


\end{frontmatter}

\newpage
\setcounter{equation}{0}
\section{\label{sec:intro} Introduction}
This note presents numerical solutions to offer asymptotic results for the Poisson distribution of order $k$,
which is a statistical distribution introduced in \cite{PhilippouGeorghiouPhilippou}.
It is a variant (or extension) of the well-known Poisson distribution.
We begin with its formal definition. 
\begin{definition}
  \label{def:pmf_Poisson_order_k}
  The Poisson distribution of order $k$ (where $k\ge1$ is an integer) and parameter $\lambda > 0$
  is an integer-valued statistical distribution with the probability mass function (pmf)
\bq
\label{eq:pmf_Poisson_order_k}
f_k(n;\lambda) = e^{-k\lambda}\sum_{n_1+2n_2+\dots+kn_k=n} \frac{\lambda^{n_1+\dots+n_k}}{n_1!\dots n_k!} \,, \qquad n=0,1,2\dots
\eq
\end{definition}
\noindent
For $k=1$ it is the standard Poisson distribution.
The notation in \cite{PhilippouMeanVar,PhilippouFibQ,GeorghiouPhilippouSaghafi,KwonPhilippou}
employs `$x$' and states that the sum in eq.~\eqref{eq:pmf_Poisson_order_k} is taken over all tuples $(x_1,\dots,x_k)$ such that $x_1 + 2x_2 + \dots + kx_k = x$.
However it is more usual for $x$ to denote a real variable and $n$ to denote an integer,
hence we adopt the symbol $n$ for better contact with the literature on statistical distributions.
As a matter of terminology, we write ``$Y_{k,\lambda} \sim \textrm{Poisson}(k,\lambda)$''
to denote a random variable $Y_{k,\lambda}$ which is Poisson distributed with order $k$ and parameter $\lambda$.

The Poisson distribution of order $k$ is a special case of the compound Poisson distribution introduced by Adelson \cite{Adelson1966},
as will be explained in Sec.~\ref{sec:distrib}.
Although exact expressions for the mean and variance of the Poisson distribution of order $k$ are known \cite{PhilippouMeanVar},
exact results for its median and mode are difficult to obtain.
In fact, even for the standard Poisson distribution, the value of the median is only known approximately.
See, e.g.~\cite{Choi,AdellJodraPoisson2005,Alm2003,ChenRubin1986} and references therein.

With reference to the Poisson distribution of order $k$,
some exact results and also upper/lower bounds for the mode have been published in \cite{PhilippouFibQ,GeorghiouPhilippouSaghafi,KwonPhilippou}.
However, this note points out that {\em asymptotic results} (for both the median and mode) exhibit simple patterns.
The relevant recurrence relations are solved to present asymptotic results for both the median and mode.
The calculations are numerical, hence the results are presented as conjectures.
The purpose of this note is to discern patterns, including expressions for exact limits and rates of convergence and (possibly sharp) upper/lower bounds,
in a sense to be made precise below.
The derivation of proofs of the observed results is left for future work.  

The structure of this paper is as follows.
Sec.~\ref{sec:distrib} presents the basic definitions and notation employed in this note.
Secs.~\ref{sec:median} and \ref{sec:mode} present asymptotic results for the median and mode, respectively.
Sec.~\ref{sec:order} analyses the relative ordering of the mean, median and mode, using the results in Secs.~\ref{sec:median} and \ref{sec:mode}.
Sec.~\ref{sec:conc} concludes.

\newpage
\setcounter{equation}{0}
\section{\label{sec:distrib}Distributions and recurrences}
\subsection{\label{sec:notation}Basic notation and definitions}
First we establish notation for the mean, variance, median and mode.
It is conventional to denote the mean by $\mu$ and the variance by $\sigma^2$.
Many authors denote the median by $\nu$, and we shall do so below.
There is no widely accepted notation for the mode.
The mode is denoted by $m$ in \cite{PhilippouFibQ,GeorghiouPhilippouSaghafi,KwonPhilippou}
and we shall do so below.

The mean and variance of the Poisson distribution of order $k$ were derived by Philippou \cite{PhilippouMeanVar}.
We denote the mean and variance by $\mu_k(\lambda)$ and $\sigma^2_k(\lambda)$, respectively.
Their values are as follows
\begin{subequations}
\begin{align}
\label{eq:mean_Poisson_orderk}
\mu_k(\lambda) &= (1+\dots+k)\lambda = \frac12\,k(k+1)\lambda \,.
\\
\label{eq:variance_Poisson_orderk}
\sigma_k^2(\lambda) &= (1^2+\dots+k^2)\lambda = \frac16\,k(k+1)(2k+1)\lambda \,.
\end{align}
\end{subequations}
The expression for the variance corrects a misprint in \cite{PhilippouMeanVar}, where the fraction was given as $\frac12$ not $\frac16$.

We define the median following the exposition in \cite{AdellJodraPoisson2005}.
It has the merit that it yields a unique value for the median.
Let $X$ be a random variable with distribution function $F(x) := P(X \le x)$, $x\in\mathbb{R}$.
Then the median is defined via
$\nu := \inf\bigl\{ z\in\mathbb{R}:\; F(z) \ge 1/2 \bigr\}$
(see eq.~(1) in \cite{AdellJodraPoisson2005})
As stated in \cite{AdellJodraPoisson2005}, this definition implies that the median of an integer valued random variable $X$ is also an integer.
Such is the case for the Poisson distribution of order $k$.

The mode is defined as the location(s) of the {\em global maximum} of the probability density (resp.~mass) function for a continuous (resp.~discrete) random variable.
The mode may not be unique.
For the Poisson distribution with parameter $\lambda$, the mode equals $\lfloor\lambda\rfloor$ if $\lambda\not\in\mathbb{N}$,
but both $\lambda-1$ and $\lambda$ are modes if $\lambda\in\mathbb{N}$.
Kwon and Philippou \cite{KwonPhilippou} published a table of values of $\lambda$ where the Poisson distribution of order $k$ has a double mode,
for $2 \le k \le 4$ and $0 < \lambda \le 2$.
Numerical studies by the author have found examples of double modes for all examined values of $k$, but to date no examples of three or more joint modes have been found.
It remains an open question if there exist values of $\lambda$ where the Poisson distribution of order $k>1$ has three or more joint modes.

\subsection{\label{sec:CompoundPoissonProof}Compound Poisson distribution}
The Poisson distribution of order $k$ is an example of a compound Poisson distribution.
Adelson \cite{Adelson1966} defined the probability generating function (pgf) of the compound Poisson distribution as follows (eq.~(1) in \cite{Adelson1966}):
\bq
\label{eq:Adelson_compoundPoisson_pgf}
f(x) = \exp\Bigl(-\sum a_i\Bigr) \exp\Bigl(\sum a_ix^i\Bigr) \,.
\eq
Adelson permitted an unlimited number of coefficients $a_i$ in eq.~\eqref{eq:Adelson_compoundPoisson_pgf}, but we limit them to $k$ terms.
Next set $a_1=\dots=a_k=\lambda$ and collect terms in powers of $x$.
The bookkeeping yields 
\bq
\begin{split}
\label{eq:Adelson_fx_series}
f(x) &= e^{-k\lambda} \sum_{n=0}^\infty \frac{\lambda^n}{n!} (x+x^2+\dots+x^k)^n
\\
&= e^{-k\lambda}\sum_{n=0}^\infty \biggl(\sum_{n_1+2n_2+\dots+kn_k=n} \frac{\lambda^{n_1+\dots+n_k}}{n_1!\dots n_k!}\biggr) \,x^n \,.
\end{split}
\eq
The coefficient of $x^n$ in eq.~\eqref{eq:Adelson_fx_series}
is the pmf of the Poisson distribution of order $k$ in eq.~\eqref{eq:pmf_Poisson_order_k}.

Let us tabulate the bookkeeping of the tuples of indices $(n_1,\dots,n_k)$ in $f_k(n;\lambda)$, for $n=0,1,2,\dots$.
For simplicity, we do so for $k=2$ only.
We obtain the array displayed in Table \ref{tb:tablehk}.
Observe that tuples with the same power of $\lambda$ appear in the same {\em column}, not the same row.
The Poisson distribution of order $k$ arranges tuples associated with a given power of $\lambda$ into columns instead of rows.
Hence for $k>1$ we cannot associate a value of $n$ with a unique power of $\lambda$.

\subsection{\label{sec:recurrences}Recurrences}
Adelson \cite{Adelson1966} denoted the coefficient of $x^j$ in eq.~\eqref{eq:Adelson_fx_series} by $R_j$ and derived a recurrence for $R_{j+1}$.
To avoid confusion of notation and to make contact with other authors, we write $n$ in place of $j+1$, and as noted above we set $a_1=\dots=a_k=\lambda$.
Then (our version of) Adelson's recurrence is (see eq.~(6) in \cite{Adelson1966})
\bq
\label{eq:recurrence_Adelson_k_terms}
\begin{split}
R_n &= \frac{\lambda}{n}\,\bigl(R_{n-1} + 2R_{n-2} + \dots + kR_{n-k} \bigr) 
\\
&= \frac{\lambda}{n}\,\sum_{j=1}^k jR_{n-j} \,.
\end{split}
\eq
Note the following:
(i) We truncated the sum on the right in eq.~\eqref{eq:recurrence_Adelson_k_terms} to $k$ terms (Adelson did not).
(ii) From \cite{Adelson1966}, the initial value is $R_0 = \exp\bigl(-\sum a_i\bigr) = e^{-k\lambda}$.
(iii) We formally define $R_n=0$ for $n<0$, to handle negative indices in the sum in eq.~\eqref{eq:recurrence_Adelson_k_terms}.

The same recurrence was derived in \cite{GeorghiouPhilippouSaghafi},
where it was written as follows (eq.~(2.3) in \cite{GeorghiouPhilippouSaghafi}, and we have changed their symbol `$x$' to $n$)
\bq
\label{eq:recurrence_GPS}
nP_n = \sum_{j=1}^k j\lambda\,P_{n-j} \,.
\eq
From \cite{GeorghiouPhilippouSaghafi}, $P_n = f_k(n;\lambda)$ (see eq.~\eqref{eq:pmf_Poisson_order_k}).
The initial value is $P_0 = e^{-k\lambda}$.  
It is not stated explicitly in \cite{GeorghiouPhilippouSaghafi}, but $P_n = 0$ for $n<0$.
Kwon and Philippou \cite{KwonPhilippou} defined $h_k(n;\lambda) = e^{k\lambda}f_k(n;\lambda)$ and expressed the recurrence as follows,
with initial value $h_k(0;\lambda) = 1$ (eq.~(6) in \cite{KwonPhilippou}):
\bq  
\label{eq:rechk}
nh_k(n;\lambda) = \begin{cases} \displaystyle \sum_{j=1}^n j\lambda\,h_k(n-j;\lambda) & \qquad (1 \le n \le k) \,,
  \\ \\
  \displaystyle \sum_{j=1}^k j\lambda\,h_k(n-j;\lambda) & \qquad (n > k) \,.
  \end{cases}
\eq
All of eqs.~\eqref{eq:recurrence_Adelson_k_terms}, \eqref{eq:recurrence_GPS} and \eqref{eq:rechk} are equivalent
(but Adelson's recurrence in \cite{Adelson1966} is not limited to a fixed number of terms).
Observe also that $h_k(n;\lambda)$ is a polynomial in $\lambda$. Note the following:
\begin{enumerate}
\item
  First, $h_k(n;\lambda)$ is a polynomial of degree $n$ because the highest power of $\lambda$
  is given by the term $\lambda^n/n!$, obtained by setting $n_1=n$ and $n_2=\dots=n_k=0$.
  Such a term always exists.
\item
  Next, $h_k(n;\lambda)$ has no constant term if $n>0$, because at least one of the $n_i$ must be nonzero.
  The lowest power of $\lambda$ is not less than $\lfloor(n/k)\rfloor$, obtained by setting $n_k = \lfloor(n/k)\rfloor$,
  and using the other $n_i$ (if necessary) to bring the value of the sum $n_1+2n_2+\dots+kn_k$ up to $n$.
\end{enumerate}

\subsection{\label{sec:KMrec}Alternative recurrence}
Kostadinova and Minkova \cite{KostadinovaMinkova2013} derived a recurrence which requires a maximum of four terms, for any $k$ and $n$.
They wrote $p_i$ and omitted explicit mention of the dependence on $k$ and $\lambda$.
We write $p_n$ for consistency with the notation of this note.
Then their recurrence is (Proposition 1 in \cite{KostadinovaMinkova2013}) 
\bq
\label{eq:KMrec}
p_n = \Bigl(2 + \frac{\lambda-2}{n}\Bigr)p_{n-1}
-\Bigl(1 -\frac{2}{n}\Bigr)p_{n-2}
-\frac{k+1}{n}\,\lambda p_{n-k-1}
+\frac{k}{n}\,\lambda p_{n-k-2} \,.
\eq
The initial value is $p_0 = e^{-k\lambda}$ and we formally define $p_n = 0$ for $n<0$.

\subsection{\label{sec:kappa}Kappa}
Numerical investigations rapidly indicate that the value of $k(k+1)/2$ is a much more important parameter than $k$ itself.
So much so that we introduce the symbol $\kappa$ and define $\kappa = k(k+1)/2$.
Much better results are obtained if we work with $\kappa\lambda$ instead of $\lambda$ itself.
Note from eq.~\eqref{eq:mean_Poisson_orderk} that the mean is given by $\mu_k(\lambda) = \kappa\lambda$.
Hence the above statement is equivalent to saying it is better to work with the mean than with $\lambda$.
For the standard Poisson distribution (i.e.~$k=1$), then $\kappa=1$ and it is known that $\mu=\lambda$.
Many authors write ``Poisson$(\mu)$'' instead of ``Poisson$(\lambda)$'' and parameterize a Poisson distribution by its mean.
However, for $k>1$, the mean and $\lambda$ are not equal.
Indeed, it might be better to characterize the Poisson distribution of order $k$ by its mean, not by $\lambda$,
and this could be stated as a finding of this note.

\newpage
\setcounter{equation}{0}
\section{\label{sec:median}Median}
We denote the median of the Poisson distribution of order $k$ by $\nu_k(\lambda)$.
Since the value of the median is necessarily an integer, we adopt the following approach.
We fix a value for $k\ge1$ and also fix the value of the median, say $\nu_*$,
and find the set of values of $\lambda$ (or equivalently the mean $=\kappa\lambda$) for which the median equals $\nu_*$.
On searching the literature, the author was pleased to discover that the same point of view was employed in
\cite{AdellJodraPoisson2005} and \cite{Alm2003} (and implicitly in \cite{Choi} and \cite{ChenRubin1986}),
all of which published fruitful results for the median of the standard Poisson distribution.

\begin{conjecture}
\label{conj:Poisson_k_median_int_n}
Fix $k\ge1$ and let $n\in\mathbb{N}$ and set $\lambda = n/\kappa$.
Numerical studies reveal that for all tested values $k\ge1$ and $n\ge\kappa$ (i.e.~$\lambda\ge1$), the value of the median is
\bq
\label{eq:Poisson_k_median_int_n}
\nu_k(n/\kappa) = n - \biggl\lfloor\frac{k+4}{8}\biggr\rfloor \,.
\eq
\end{conjecture}
\noindent
Note the following:
\begin{enumerate}
\item  
  For $k=1$, then $\kappa=1$ and eq.~\eqref{eq:Poisson_k_median_int_n} simplifies to the known value $\nu_1(n) = n$ (see Theorem 1 in \cite{AdellJodraPoisson2005}).  
\item
  For $k>1$, the value of $n$ and the median are not always equal, and eq.~\eqref{eq:Poisson_k_median_int_n} establishes the relation between them.
\item
  For smaller values $n<\kappa$, eq.~\eqref{eq:Poisson_k_median_int_n} is not always valid.
  In particular eq.~\eqref{eq:Poisson_k_median_int_n} cannot be correct for $n < \lfloor(k+4)/8\rfloor$.  
  It is an open question what is the smallest value of $n$ for which eq.~\eqref{eq:Poisson_k_median_int_n} is valid.
\item  
  Hence eq.~\eqref{eq:Poisson_k_median_int_n} is an asymptotic result for sufficiently large values of $n$,
  where in this context ``sufficiently large'' means $n \ge \kappa$, i.e.~$\lambda\ge1$.
\item
  Observe that eq.~\eqref{eq:Poisson_k_median_int_n} implies that for $n \ge \kappa$,
  a unit increase in the value of $n$ yields a unit increase in the value of the median.
  This is not always the case if $n<\kappa$.  
\end{enumerate}
Before proceeding with more numerical studies, we derive a simple result, viz.~the largest value of $\lambda$, say $\lambda_*$, such that the median equals zero.
\begin{theorem}
\label{thm:Poisson_k_median_lam_median_zero}
The Poisson distribution of order $k$ has a median value of zero if and only if $\lambda \le \lambda_*$, where $\lambda_* = (\ln 2)/k$.
\end{theorem}
\begin{proof}
The equation to solve is simply $e^{-k\lambda_*}=\frac12$, hence $\lambda_* = (\ln 2)/k$.
It is easily seen that this is a necessary and sufficient condition.
\end{proof}
\begin{corollary}
\label{corr:Poisson_k_median_klam_median_zero}
The corresponding value of $\kappa\lambda_*$ is $\frac12(k+1)\ln2$.
Note that $\kappa\lambda_*>1$ for all $k>1$, e.g.~$\kappa\lambda_* = \frac32\ln2 \simeq 1.0397$ for $k=2$.
Hence the median of the Poisson distribution of order $k$ equals zero for $\kappa\lambda=1$ for all $k>1$,
and in fact $\nu_k(n/\kappa)=0$ for all $n \le n_*$ where $n_* = \bigl\lfloor \frac12(k+1)\ln2 \bigr\rfloor$ \,.
\end{corollary}
\noindent
As the value of $\lambda$ increases from zero, the value of the median increases in unit steps (starting from zero),
as is necessarily the case because for fixed $n$, the probability $P(Y_{k,\lambda} \le n)$ is a continuous function of $\lambda$
(where $Y_{k,\lambda} \sim \textrm{Poisson}(k,\lambda)$).
However, the locations of the steps (i.e.~the values of $\lambda$, or better $\kappa\lambda$, where the value of the median increases) are irregular,
and attain a simple pattern only for $n \ge \kappa$.
Fig.~\ref{fig:median_mode_k3} displays a plot of the median (solid line) as a function of $\kappa\lambda$ (which equals the mean)
for the Poisson distribution of order $3$ for $0 < \kappa\lambda \le 15$.
The mode is also displayed (dashed line) and will be discussed in Sec.~\ref{sec:mode}.
Then $\kappa=6$ for $k=3$.
Observe that the spacing of the steps in the median is irregular for $n<\kappa$, i.e.~$\lambda<1$,
but for $n\ge\kappa$, i.e. $n\ge6$, the median increases by one unit for each unit increase in $n$.
Note also that a step in the value of the median indicates a value of $\lambda$ such that $P(Y_{k,\lambda} \le \nu)=\frac12$ exactly.

In the rest of this section, we restrict attention to $n\ge\kappa$.
Then, for fixed $n\ge\kappa$, we determine the interval of values of $\lambda$ such that the median is given by $\nu_k(n/\kappa)$ in eq.~\eqref{eq:Poisson_k_median_int_n}.
To lay the groundwork for later developments, we begin with $k=1$ and define a set $\{\alpha_0,\alpha_1,\dots\}$ with the property that
if $\lambda\in(\ln2,\infty)$ and the value of the median is $\nu = n (\ge1)$, then the value of the mean $\mu$ lies in the interval
$\mu \in (\alpha_{n-1}, \alpha_n]$.
From \cite{Choi}, it is known that $\alpha_0=\ln2$ and that asymptotically 
\bq
\label{eq:Poisson_asymp_alpha_n}
\alpha_n = n + \frac23 + \frac{8}{405\, n} - \frac{64}{5103\, n^2} + \frac{2^7\cdot23}{3^9\cdot5^2\, n^3} + O\Bigl(\frac{1}{n^4}\Bigr) \,.
\eq
Numerical studies have verified eq.~\eqref{eq:Poisson_asymp_alpha_n} to high precision, which should be interpreted as a benchmark test of the numerical calculations.
We now study the behavior of the mean $\mu_k(\lambda)$ for $k>1$.
Analogous to the analysis for $k=1$, we define a set $\{\alpha_{k,n}|n\ge\kappa\}$ as follows:
for $n\ge\kappa+1$, the median $\nu_k(n/\kappa)$ is given by eq.~\eqref{eq:Poisson_k_median_int_n} and the value of the mean $\mu_k(\lambda)$ lies in the interval
$\mu_k(\lambda) \in (\alpha_{k,n-1}, \alpha_{k,n}]$.
By construction, if $\mu_k(\lambda) = \alpha_{k,n}$, then $P(Y_{k,\lambda} \le n)=\frac12$ exactly
(where $Y_{k,\lambda} \sim \textrm{Poisson}(k,\lambda)$).

Unlike the case $k=1$, where $\alpha_n$ is well-defined for all $n\ge0$, for $k>1$ we require $n\ge\kappa$ for $\alpha_{k,n}$ to be well-defined.
We saw in Corollary \ref{corr:Poisson_k_median_klam_median_zero} that the median is zero for all $0 \le n \le \bigl\lfloor \frac12(k+1)\ln2 \bigr\rfloor$.
Hence for $n < \kappa$ there may be several values of $n$ which share the same value of the median.
Numerical calculations indicate that the asymptotic value of $\alpha_{k,n}$ is
\bq
\label{eq:Poisson_k_asymp_alpha_kn2}
\alpha_{k,n} = n + \textrm{frac}\Bigl(\frac{k+4}{8}\Bigr) +\frac{k}{8(2k+1)} + A_{k,n} \,.
\eq
Here
\bq
\label{eq:Poisson_k_median_Akn}
A_{k,n} = \biggl(\frac{3\kappa}{349} + \frac{13}{1000}\biggr)\frac{1}{n} + \frac{13}{1500}\biggl(\biggl\lfloor\frac{k+4}{8}\biggr\rfloor -3\biggr)\,\frac{\kappa}{n^2} +\cdots
\eq
Note the following:
\begin{enumerate}
\item  
  The terms which depend only on $k$ seem to be exact.
  In particular, for $k=1$ eq.~\eqref{eq:Poisson_k_asymp_alpha_kn2} yields
  $\textrm{frac}((k+4)/8) + k/(8(2k+1)) = 5/8 + 1/24 = 2/3$,
  which is the correct value (see \cite{Choi} and eq.~\eqref{eq:Poisson_asymp_alpha_n}).
\item
  The dependence on $n$ in $A_{k,n}$ is approximate, despite its seeming detail.
  The leading term $3/349$ is $1/(116+\frac13)$, which was chosen to give the best fit to the numerical data.
  Similarly, the value of the term $13/(1000 n)$ is small for $n\ge\kappa$ and is approximate.
  For example, for $k=1$ it is known that the exact value of the coefficient of the $O(1/n)$ term is $8/405$ \cite{Choi}, but $(3/349)+(13/1000) \ne 8/405$.
  The residuals in the numerical data (the difference between the numerical results and the fit using eq.~\eqref{eq:Poisson_k_asymp_alpha_kn2})
  exhibit additional dependence of $O(1/n^2)$, but the residual terms are small and difficult to quantify.
\item
  For values (i) $k=10$ and $\lambda \in [100,1000]$ or (ii) $\lambda=100$ and $k \in [10,100]$,
  the difference between eq.~\eqref{eq:Poisson_k_asymp_alpha_kn2} and the numerical data is $O(10^{-8})$.
  Even for $k=2$ and $\lambda \in [50,1000]$ the difference is $O(10^{-7})$.
\item
  Notice that the expression ``$(k+4)/8$'' appears frequently in formulas pertaining to the median. The reason for this is not known.
\end{enumerate}
The asymptotic upper bound for the difference between the mean and the median is (again, only for $k>1$, and better for $k\ge10$)
\bq
\label{eq:Poisson_k_asymp_mean_minus_median}
\begin{split}
\alpha_{k,n} - \nu_k(n/\kappa) &= \frac{k+4}{8} +\frac{k}{8(2k+1)} + A_{k,n} 
\\
&= \frac{2k+9}{16} -\frac{1}{16(2k+1)} + A_{k,n} \,.
\end{split}
\eq
Hence we state the following conjectures.
\begin{conjecture}
\label{conj:Poisson_k_alpha_kn_terms_in_k_only}
The terms in eqs.~\eqref{eq:Poisson_k_asymp_alpha_kn2} and \eqref{eq:Poisson_k_asymp_mean_minus_median} which depend only on $k$ are exact.
\end{conjecture}
\noindent
Next, for $k=1$, it is known that $\alpha_n-n \in(0,1)$ for all values $n\ge0$ (see \cite{AdellJodraPoisson2005}).
We state the following bound for $k>1$.
\begin{conjecture}
\label{conj:Poisson_k_bounds_alpha_kn}
For fixed $k>1$ and $n \ge \kappa$, we claim $\alpha_{k,n}-n \in(0,1)$.
This is not always true if $n<\kappa$.
\end{conjecture}
\noindent
Next, for $k=1$, it is known that the value of $\alpha_n -n$ is a decreasing function of $n$ for all $n\ge0$.
This was conjectured by Chen and Rubin \cite{ChenRubin1986} and proved by Alm \cite{Alm2003}.
It is not so clear-cut for $k>1$.
For $k=2$ and $\kappa \le n < 3\kappa$, i.e.~$3 \le n < 9$, it is {\em false} that the value of $\alpha_{k,n} -n$ is a decreasing function of $n$.
For $k\in[3,6]$, there are also problems of non-monotonicity if $n < 2\kappa$.
We make the following conjecture for $k>1$.
\begin{conjecture}
\label{conj:Poisson_k_alpha_kn_fix_k_dec_n}
For (i) $k=2$ and $n\ge3\kappa$ (i.e.~$n\ge9$)
or (ii) $k\in[3,6]$ and $n\ge2\kappa$ 
or (iii) $k\ge7$ and $n\ge\kappa$,
the value of $\alpha_{k,n} -n$ is a decreasing function of $n$.
\end{conjecture}
\noindent
For $k=1$, sharp upper and lower bounds for the value of $\alpha_n-n$ were conjectured by Chen and Rubin \cite{ChenRubin1986} and proved by Choi \cite{Choi}.
We state the following limit and lower bound for $k>1$.
\begin{conjecture}
\label{conj:Poisson_k_limit_alpha_kn}
If the conditions in Conjecture \ref{conj:Poisson_k_alpha_kn_fix_k_dec_n} are satisfied, so that $\alpha_{k,n} - n$ is a decreasing function of $n$,
then for fixed $k>1$, the limit as $n\to\infty$ is as follows,
and moreover we conjecture it is a sharp lower bound for the value of $\alpha_{k,n} - n$.
\bq
\label{eq:Poisson_k_limit_alpha_kn}
\lim_{n\to\infty} (\alpha_{k,n} - n) = \textrm{\rm frac}\Bigl(\frac{k+4}{8}\Bigr) +\frac{k}{8(2k+1)} \,.
\eq
A conjecture for a sharp upper bound for the value of $\alpha_{k,n} - n$ is not known.
\end{conjecture}
\noindent
For $k=1$, sharp bounds on the difference between the mean and median are known \cite{Choi}.
We state the following for $k>1$.
\begin{conjecture}
\label{conj:Poisson_k_limit_mean_minus_median}
If the conditions in Conjecture \ref{conj:Poisson_k_alpha_kn_fix_k_dec_n} are satisfied, so that $\alpha_{k,n} - n$ is a decreasing function of $n$,
then for fixed $k>1$, the limit as $n\to\infty$ is as follows
and moreover we conjecture it is a sharp lower bound for the value of $\alpha_{k,n} - \nu_k(n/\kappa)$.
\bq
\label{eq:Poisson_k_limit_mean_minus_median}
\lim_{n\to\infty} (\alpha_{k,n} - \nu_k(n/\kappa)) = \frac{2k+9}{16} -\frac{1}{16(2k+1)} \,.
\eq
A conjecture for a sharp upper bound for the value of $\alpha_{k,n} - \nu_k(n/\kappa)$ is not known.
\end{conjecture}
\begin{conjecture}
\label{conj:Poisson_k_limit_mean_minus_median_k_infty_n_infty}
We can rewrite eq.~\eqref{eq:Poisson_k_limit_mean_minus_median}
and conjecture a further limit as both $k\to\infty$ and $n\to\infty$:
\bq
\label{eq:Poisson_k_limit_mean_minus_median_k_infty_n_infty}
\lim_{k\to\infty} \biggl\{\Bigl[\,\lim_{n\to\infty} (\alpha_{k,n} - \nu_k(n/\kappa))\Bigr] - \frac{2k+9}{16}\,\biggr\} = 0 \,.
\eq
\end{conjecture}

\newpage
\setcounter{equation}{0}
\section{\label{sec:mode}Mode}
The mode of the Poisson distribution of order $k$ is denoted by $m_{k,\lambda}$ in \cite{PhilippouFibQ,GeorghiouPhilippouSaghafi,KwonPhilippou}
but for consistency with the notation in Sec.~\ref{sec:median}, we denote the mode by $m_k(\lambda)$ below.
Since the value of the mode is necessarily an integer, we employ the same policy as in Sec.~\ref{sec:median}:
we fix a value for $k\ge1$ and also fix the value of the mode, say $m_*$,
and find the set of values of $\lambda$ (or equivalently the mean $=\kappa\lambda$) for which the mode equals $m_*$.

\begin{conjecture}
\label{conj:Poisson_k_mode_int_n}
Fix $k\ge1$ and let $n\in\mathbb{N}$ and set $\lambda = n/\kappa$.
Numerical studies reveal that for all tested values $k\ge1$ and $n\ge2\kappa$, the value of the mode is
\bq
\label{eq:Poisson_k_mode_int_n}
m_k(n/\kappa) = n - \biggl\lfloor\frac{3k+5}{8}\biggr\rfloor \,.
\eq
\end{conjecture}
\noindent
Note the following:
\begin{enumerate}
\item
  Unlike the results for the median in Sec.~\ref{sec:median}, for the mode we require $n\ge2\kappa$.
\item
  Before proceeding further, we discuss the case $k=1$, the standard Poisson distribution.
  For $k=1$, then $\kappa=1$ and eq.~\eqref{eq:Poisson_k_mode_int_n} simplifies to $m_1(n) = n-1$.
  However, in this situation, i.e.~when the value of $\lambda$ is an integer $(=n)$, the Poisson distribution has a double mode at $n-1$ and $n$.
  This scenario applies only for $k=1$.
  For $k>1$, the numerical results indicate that a joint mode does not occur for any tested value $\lambda = n/\kappa$, for $n \ge 2\kappa$.
  This fact causes the case $k=1$ not to fit the general pattern for $k>1$, as will be explained below.
\item
  For $k>1$, eq.~\eqref{eq:Poisson_k_mode_int_n} establishes the relation between $n$ and the value of the (unique) mode.
\item
  For smaller values $n<2\kappa$, eq.~\eqref{eq:Poisson_k_mode_int_n} is not always valid.
  In particular eq.~\eqref{eq:Poisson_k_mode_int_n} cannot be correct for $n < \lfloor(3k+5)/8\rfloor$.  
  It is an open question what is the smallest value of $n$ for which eq.~\eqref{eq:Poisson_k_mode_int_n} is valid.
\item  
  Hence eq.~\eqref{eq:Poisson_k_mode_int_n} is an asymptotic result for sufficiently large values of $n$,
  where in this context ``sufficiently large'' means $n \ge 2\kappa$, i.e.~$\lambda\ge2$.
\item
  Observe that eq.~\eqref{eq:Poisson_k_mode_int_n} implies that for $n \ge 2\kappa$,
  a unit increase in the value of $n$ yields a unit increase in the value of the mode.
  This is not always the case if $n<2\kappa$.
  Recall Fig.~\ref{fig:median_mode_k3}, which displayed a plot of both the median (solid line) and the mode (dashed line) as a function of $\kappa\lambda$
  for the Poisson distribution of order $3$ for $0 < \kappa\lambda \le 15$.
  Then $\kappa=6$ for $k=3$.
  Unlike the median, the mode does not always increase in unit steps and the spacing of the steps is irregular for $n<2\kappa$, i.e.~$\lambda<2$,
  but for $n\ge2\kappa$, i.e. $n\ge12$, the mode increases by one unit for each unit increase in $n$.
  Note also that a step in the value of the mode indicates a bimodal distribution (the mode has two values for the same value of $\lambda$).
\end{enumerate}
Before proceeding with asymptotic formulas for large $k$ and $n$,
we note the following results for the Poisson distribution of order $k$ to have a unique mode of zero.
\begin{enumerate}
\item
Philippou \cite{PhilippouFibQ} proved that for all $k\ge1$,
the Poisson distribution of order $k$ has a unique mode of zero if $0 < \lambda < 2/(k(k+1))$, i.e.~$\kappa\lambda \in(0,1)$.
\item
The above is a sufficient but not necessary condition for the Poisson distribution of order $k$ to have a unique mode of zero.
For $k=2$, Philippou \cite{PhilippouFibQ} also proved that the Poisson distribution of order $2$ has a unique mode of zero for $0 < \lambda < \sqrt{3}-1$.
However, $\kappa=3$ for $k=2$ and $\sqrt{3}-1 \simeq 0.732 > \frac13$.
\item
Furthermore, when $\lambda = \sqrt{3}-1$, the Poisson distribution of order $2$ has double modes $0$ and $2$,
i.e.~the mode increases by {\em two} units as the value of $\lambda$ increases across $\sqrt{3}-1$.
Unlike the median, the mode does not always increase in unit steps as the value of $\lambda$ increases continuously from zero.
(But it seems to do so for $n\ge2\kappa$, see eq.~\eqref{eq:Poisson_k_mode_int_n}.)
\item  
  Kwon and Philippou \cite{KwonPhilippou} published a table where the Poisson distribution of order $k$ has double modes,
  for values of $\lambda$ in the interval $0 < \lambda < 2$ and $k = 2,3,4$.
  The first double modes to appear are at $m_k(\lambda) = 0$ and $k$, for $k=2,3,4$.
\item
  It is natural to conjecture that the first double mode occurs at $0$ and $k$ for all $k \ge 1$, but this is false.
  Numerical calculations by the author found that the first double mode occurs at $0$ and $k$ for $1 \le k \le 14$.
  However, for $k=15$, the first double mode occurs at $0$ and $25$, and are never $0$ and $k$ for all tested values $k\ge15$.
  Fig.~\ref{fig:histk15_dbl_mode} displays a plot of the histogram of $h_k(n;\lambda)$ for the Poisson distribution of order $k$
  for $\lambda \simeq 0.25023$, where there is a double mode at $0$ and $25$.
  There is a local maximum at $n=k=15$, but it is not a global maximum because $h_{15}(n;\lambda) \simeq 0.9945$ for $n=15$.
  Observe also that the histogram is not smooth, which complicates the derivation of a formula for the location of the mode.
  The histogram plot of $h_k(n;\lambda)$ is much smoother for $n\ge2\kappa$, i.e.~$\lambda\ge2$.
  Fig.~\ref{fig:histk15_lam2} displays a plot of the histogram of $h_k(n;\lambda)$ for the Poisson distribution of order $15$ for $\lambda=2$.
  Then $\kappa=120$ and the mean is $\mu=\kappa\lambda=240$ and the (unique) mode is at $m=234$, in agreement with eq.~\eqref{eq:Poisson_k_mode_int_n}.
\item
  Hence unlike the median, which is zero if and only if $\lambda \le (\ln2)/k$,
  there is no simple formula for $\sup(\lambda)$ (or $\sup(\kappa\lambda)$) such that the Poisson distribution of order $k$ has a unique mode of zero.
\end{enumerate}
However, numerical studies indicate that the value of $(\kappa\lambda)_0=\sup(\kappa\lambda)$ (for `mode of zero')
such that the Poisson distribution of order $k$ has a unique mode of zero
has an asymptotic value $(\kappa\lambda)_0 \propto k^{1.125}$, i.e.~an exponent of $9/8$.
A logarithmic graph of $(\kappa\lambda)_0$ is plotted in Fig.~\ref{fig:lamfirst_mode}, for $1 \le k \le 10^4$.
The dashed line is ($\textrm{constant}\times k^{1.125}$). It indicates a good fit for $k \gtrsim 100$.
Since $\kappa=k(k+1)/2$, the corresponding value $\sup(\lambda)$ itself scales as $k^{1.125-2} = k^{-0.875} = k^{-7/8}$.
Once again, this is a numerical observation, not proved.

This leads us to a companion observation: {\em if the median is zero then the mode is also zero.}
Recall that the median is zero if and only if $\lambda \le (\ln2)/k$.
Numerical studies confirm that this value of $\lambda$ is small enough that the mode is also zero.
This is a larger upper bound than $2/(k(k+1))$, derived in \cite{PhilippouFibQ}.
Actually, the findings in Fig.~\ref{fig:lamfirst_mode} indicate that a power of $k^{-1}$ is not the optimum.
However, for a higher power such as $k^{-7/8}$ the mode is zero but the median is not necessarily zero.

We next comment on two theorems by Georghiou, Philippou and Saghafi \cite{GeorghiouPhilippouSaghafi}.
\begin{enumerate}
\item
  The above authors proved that for $\lambda\in\mathbb{N}$ and $2\le k \le 5$, the Poisson distribution of order $k$ has a unique mode given as follows
(Theorem 2.2 in \cite{GeorghiouPhilippouSaghafi})
\bq
\label{eq:mode_Georghiou_etal_Thm2.2}
m_k(\lambda) = \frac{k(k+1)}{2}\,\lambda - \biggl\lfloor \frac{k}{2} \biggr\rfloor \,.
\eq
We compare eqs.~\eqref{eq:Poisson_k_mode_int_n} and \eqref{eq:mode_Georghiou_etal_Thm2.2}.
First, $\lfloor k/2 \rfloor = \lfloor (3k+5)/8 \rfloor$ for $2 \le k \le 5$
(but eq.~\eqref{eq:Poisson_k_mode_int_n} is not restricted to the interval $2 \le k \le 5$).
Next, eq.~\eqref{eq:mode_Georghiou_etal_Thm2.2} was stated and proved only for integers  $\lambda\in\mathbb{N}$
but eq.~\eqref{eq:Poisson_k_mode_int_n} is proposed {\em (it has not been proved)} for the superset $\kappa\lambda\in\mathbb{N}$,
i.e.~$\lambda=n/\kappa$ for $n\in\mathbb{N}$, subject to the restriction $\lambda\ge2$.
Numerical calculations show that for $2 \le k \le 5$, eq.~\eqref{eq:Poisson_k_mode_int_n} is also valid for $\lambda=1$.
Hence eqs.~\eqref{eq:Poisson_k_mode_int_n} and \eqref{eq:mode_Georghiou_etal_Thm2.2}
are equal for the conditions stated in Theorem 2.2 in \cite{GeorghiouPhilippouSaghafi}.

\item
  The above authors also proved the following bounds.
  For any integer $k\ge1$ and real $\lambda>0$, the mode of the Poisson distribution of order $k$ satisfies the inequalities
  (Theorem 2.1 in \cite{GeorghiouPhilippouSaghafi}):
\bq  
\label{eq:mode_Georghiou_etal_Thm2.1}
\bigl\lfloor \kappa\lambda \bigr\rfloor - \kappa + 1 - \delta_{k,1} \le m_k(\lambda) \le \bigl\lfloor \kappa\lambda \bigr\rfloor \,.
\eq
The expression in eq.~\eqref{eq:Poisson_k_mode_int_n} satisfies the bounds in eq.~\eqref{eq:mode_Georghiou_etal_Thm2.1}.
The bounds in eq.~\eqref{eq:mode_Georghiou_etal_Thm2.1} are trivially attained for $k=1$, whenever the value of $\lambda$ is a positive integer.
For $k>1$, the upper bound is sharp because it is attained whenever $0 < \lambda < 1/\kappa$, because $\lfloor \kappa\lambda \rfloor=0$
and as Philippou \cite{PhilippouFibQ} proved, the value of the mode is then zero.
The lower bound is not sharp because its value is negative when $\lfloor \kappa\lambda \rfloor < \kappa - 1 + \delta_{k,1}$
(this was noted in \cite{GeorghiouPhilippouSaghafi}).
We sharpen it by imposing a floor of zero.
The revised lower bound is then also attained whenever $0 < \lambda < 1/\kappa$, because the mode is zero and the lower bound is also zero.
With this modification, we obtain sharp bounds for the mode.
\end{enumerate}
\begin{proposition}
  (Rewording of lower bound of Theorem 2.1 in \cite{GeorghiouPhilippouSaghafi}.)
  For any integer $k\ge1$ and real $\lambda>0$, the mode of the Poisson distribution of order $k$ satisfies the sharp inequalities
\bq  
\label{eq:mode_Georghiou_etal_Thm2.1_me}
\max\bigl\{0, \bigl\lfloor \kappa\lambda \bigr\rfloor - \kappa + 1 - \delta_{k,1} \bigr\} \le m_k(\lambda) \le \bigl\lfloor \kappa\lambda \bigr\rfloor \,.
\eq
\end{proposition}

In the rest of this section, we restrict attention to $n\ge2\kappa$.
Then, for fixed $n\ge2\kappa$, we determine the interval of values of $\lambda$ such that the mode is given by $m_k(n/\kappa)$ in eq.~\eqref{eq:Poisson_k_mode_int_n}.
Analogous to the analysis for the median, we define a set $\{\beta_{k,n}|n\ge2\kappa\}$ as follows:
for $n\ge2\kappa+1$, the mode $m_k(n/\kappa)$ is given by eq.~\eqref{eq:Poisson_k_mode_int_n} and the value of the mean $\mu_k(\lambda)$ lies in the interval
$\mu_k(\lambda) \in (\beta_{k,n-1}, \beta_{k,n}]$.
By construction, if $\mu_k(\lambda) = \beta_{k,n}$, the Poisson distribution of order $k$ has a double mode at $m_k(n/\kappa)$ and $m_k(n/\kappa)+1$.
Hence, for any $k > 1$, there is a denumerable infinity of values of $\lambda$ for which the Poisson distribution of order $k$ has a double mode,
and (for $n\ge2\kappa$) the double mode consists of consecutive integers.
{\em (This is again a numerical observation, not formally proved.)}
As already stated, it is not known if the Poisson distribution of order $k$ has three or more joint modes.

For $k=1$, the exact solution is $\beta_{1,n} = n$, because the double modes occur at $\lambda = n-1$ and $n$.
This fact complicates matters in the sense the case $k=1$ does not fit into the patterns below for $k>1$.
This fact should be borne in mind when reading the asymptotic results and conjectures below.
In particular, for $k>1$ we require $n\ge2\kappa$ for $\beta_{k,n}$ to be well-defined.
We saw above that the mode is zero for all integers $0 \le n \le \lfloor (\kappa\lambda)_0\rfloor$.
Hence for $n < 2\kappa$ there may be several values of $n$ which share the same value of the mode.
Numerical calculations indicate that the asymptotic value of $\beta_{k,n}$ is
\bq
\label{eq:Poisson_k_asymp_beta_kn2}
\beta_{k,n} = n + \textrm{frac}\Bigl(\frac{3k+5}{8}\Bigr) +\frac{k-1}{8(2k+1)} + B_{k,n} \,.
\eq
Here
\bq
\label{eq:Poisson_k_mode_Bkn}
B_{k,n} = \biggl(\frac{\kappa}{16 +\frac89} - \frac{1}{13+\frac23}\biggr)\frac{1}{n}
+ \biggl\lfloor\frac{3k+5}{8}\biggr\rfloor \,\frac{3\kappa}{50n^2} +\cdots
\eq
Note the following:
\begin{enumerate}
\item
  The terms which depend only on $k$ seem to be exact.
  In particular, for $k=1$ eq.~\eqref{eq:Poisson_k_asymp_beta_kn2} yields
  $\textrm{frac}((3k+5)/8) + (k-1)/(8(2k+1)) = 0$, which is the correct value.
\item
  The dependence on $n$ in $B_{k,n}$ is approximate, despite its seeming detail.
  The denominator numbers $16\frac89$ and $13\frac23$ were chosen to give the best fit to the numerical data.
  For example, for $k=1$ it is known that $\beta_{1,n}=0$, hence the expression for $B_{1,n}$ cannot be correct.
  The residuals in the numerical data (the difference between the numerical results and the fit using eq.~\eqref{eq:Poisson_k_asymp_beta_kn2})
  exhibit additional dependence of $O(1/n^2)$, but the residual terms are small and difficult to quantify.
\item
  For values (i) $k=10$ and $\lambda \in [100,1000]$ or (ii) $\lambda=100$ and $k \in [10,100]$,
  the difference between eq.~\eqref{eq:Poisson_k_asymp_beta_kn2} and the numerical data is $O(10^{-6})$.
  Even for $k=2$ and $\lambda \in [50,1000]$ the difference is $O(10^{-5})$.
\item
  Notice that the expression ``$(3k+5)/8$'' appears frequently in formulas pertaining to the mode. The reason for this is not known.
\end{enumerate}
The asymptotic upper bound for the difference between the mean and the mode is (again, only for $k>1$, and better for $k\ge10$)
\bq
\label{eq:Poisson_k_asymp_mean_minus_mode}
\begin{split}
\beta_{k,n} - m_k(n/\kappa) &= \frac{3k+5}{8} +\frac{k-1}{8(2k+1)} + B_{k,n}
\\
&= \frac{6k+11}{16} -\frac{3}{16(2k+1)} + B_{k,n} \,.
\end{split}
\eq
Hence we state the following conjectures.
\begin{conjecture}
\label{conj:Poisson_k_beta_kn_terms_in_k_only}
The terms in eqs.~\eqref{eq:Poisson_k_asymp_beta_kn2} and \eqref{eq:Poisson_k_asymp_mean_minus_mode} which depend only on $k$ are exact.
\end{conjecture}
\noindent
We state the following bound for $k>1$.
\begin{conjecture}
\label{conj:Poisson_k_bounds_beta_kn}
For fixed $k>1$ and $n \ge 2\kappa$, we claim $\beta_{k,n}-n \in(0,1)$.
This is not always true if $n<2\kappa$.
\end{conjecture}
\noindent
Next, for $k=1$, it is known that $\beta_{1,n} -n=0$ for all $n\ge0$.
It is not so clear-cut for $k>1$.
For $k=2$ and $2\kappa \le n < 5\kappa$, i.e.~$6 \le n < 15$, it is {\em false} that the value of $\beta_{k,n} -n$ is a decreasing function of $n$.
For $k=3$ and $k=4$, there are also problems of non-monotonicity if $n < 3\kappa$.
We make the following conjecture for $k>1$.
\begin{conjecture}
\label{conj:Poisson_k_beta_kn_fix_k_dec_n}
For (i) $k=2$ and $n\ge5\kappa$ (i.e.~$n\ge15$)
or (ii) $k\in[3,4]$ and $n\ge3\kappa$ 
or (iii) $k\ge5$ and $n\ge2\kappa$,
the value of $\beta_{k,n} -n$ is a decreasing function of $n$.
\end{conjecture}
\noindent
Again for $k=1$, the exact result $\beta_{1,n}-n=0$ is known.
We state the following limit and lower bound for $k>1$.
\begin{conjecture}
\label{conj:Poisson_k_limit_beta_kn}
If the conditions in Conjecture \ref{conj:Poisson_k_beta_kn_fix_k_dec_n} are satisfied, so that $\beta_{k,n} - n$ is a decreasing function of $n$,
then for fixed $k>1$, the limit as $n\to\infty$ is as follows,
and moreover we conjecture it is a sharp lower bound for the value of $\beta_{k,n} - n$.
\bq
\label{eq:Poisson_k_limit_beta_kn}
\lim_{n\to\infty} (\beta_{k,n} - n) = \textrm{\rm frac}\Bigl(\frac{3k+5}{8}\Bigr) +\frac{k-1}{8(2k+1)} \,.
\eq
A conjecture for a sharp upper bound for the value of $\beta_{k,n} - n$ is not known.
\end{conjecture}
\noindent
For $k=1$, the difference between the mean and mode is known exactly.
We state the following for $k>1$.
\begin{conjecture}
\label{conj:Poisson_k_limit_mean_minus_mode}
If the conditions in Conjecture \ref{conj:Poisson_k_beta_kn_fix_k_dec_n} are satisfied, so that $\beta_{k,n} - n$ is a decreasing function of $n$,
then for fixed $k>1$, the limit as $n\to\infty$ is as follows
and moreover we conjecture it is a sharp lower bound for the value of $\beta_{k,n} - m_k(n/\kappa)$.
\bq
\label{eq:Poisson_k_limit_mean_minus_mode}
\lim_{n\to\infty} (\beta_{k,n} - m_k(n/\kappa)) = \frac{6k+11}{16} -\frac{3}{16(2k+1)} \,.
\eq
A conjecture for a sharp upper bound for the value of $\beta_{k,n} - m_k(n/\kappa)$ is not known.
\end{conjecture}
\begin{conjecture}
\label{conj:Poisson_k_limit_mean_minus_mode_k_infty_n_infty}
We can rewrite eq.~\eqref{eq:Poisson_k_limit_mean_minus_mode}
and conjecture a further limit as both $k\to\infty$ and $n\to\infty$:
\bq
\label{eq:Poisson_k_limit_mean_minus_mode_k_infty_n_infty}
\lim_{k\to\infty} \biggl\{\Bigl[\,\lim_{n\to\infty} (\beta_{k,n} - m_k(n/\kappa))\Bigr] - \frac{6k+11}{16}\,\biggr\} = 0 \,.
\eq
\end{conjecture}

\newpage
\section{\label{sec:order}Ordering of mean, median and mode}
Although not true in all cases, it is recognized that many statistical distributions
exhibit the ordering (i) mode $<$ median $<$ mean, else (ii) mean $<$ median $<$ mode.
Such an ordering is difficult to quantify for the Poisson distribution of order $k$.
Even for the standard Poisson distribution, the mean can be either lower or higher than the median.
(See \cite{Choi}: if the median equals $n$, the value of the mean $\mu$ is bounded by $n-\frac13 < \mu < n+\ln2$.)

We therefore treat only integer values for the mean.
We {\em first} fix the value of the mean to equal an integer $n>0$, i.e.~we fix $\lambda=n/\kappa$,
and {\em then} calculate the corresponding values of the median $\nu_k(n/\kappa)$ and mode $m_k(n/\kappa)$.
With this policy, the Poisson distribution of order $k$ displays the first ordering above.
For $k=2$ or $3$ and $n\ge2\kappa$, if the mean equals $n$, i.e.~$\lambda=n/\kappa$,
then the median (see eq.~\eqref{eq:Poisson_k_median_int_n}) and the mode (see eq.~\eqref{eq:Poisson_k_mode_int_n}) obey the ordering
\bq
\label{eq:ordering1}
\{ m_k(n/\kappa),\, \nu_k(n/\kappa),\, \mu_k(n/\kappa) \} = \{n-1,\,n,\,n\} \,.
\eq
For $k \ge 4$ (and $n\ge2\kappa$) it is a strict ordering
\bq
\label{eq:ordering2}
  m_k(n/\kappa) < \nu_k(n/\kappa) < \mu_k(n/\kappa) \,.
\eq
The restriction $n\ge2\kappa$ is required to justify eq.~\eqref{eq:Poisson_k_mode_int_n} for the mode.
Note of course that both eqs.~\eqref{eq:Poisson_k_median_int_n} and \eqref{eq:Poisson_k_mode_int_n} are conjectured expressions.

We have seen that for sufficiently small values of $\lambda$, both the median and mode equal zero.
A Monte Carlo search was performed to calculate the median and mode for values $2 \le k \le 1000$ and $1 \le n < 2\kappa$
(and the value $\lambda=n/\kappa$ was employed in the test, so the mean was $n$).
Note that for $n<\kappa$ we do not have an analytical expression for the median
and for $n<2\kappa$ we do not have an analytical expression for the mode,
hence their values were computed {\em ab initio}.
The values of the mode, median and mean obey the ordering
\bq
\label{eq:ordering3}
m_k(n/\kappa) \le \nu_k(n/\kappa) \le \mu_k(n/\kappa) \qquad (1 \le n < 2\kappa)\,.
\eq
In contrast to eqs.~\eqref{eq:ordering1} and \eqref{eq:ordering2}, we require ``mode $\le$ median''
because it is possible for the median and mode to be equal if $n < 2\kappa$.
In particular they could both be zero.

\newpage
\section{\label{sec:conc}Conclusion}
There are several points to note about the Poisson distribution of order $k$.
First, all of the calculations in this note are numerical, hence all the results are phrased as conjectures.
Several conjectures were stated for both the median and the mode of the Poisson distribution of order $k$.
These included conjectures for exact results, limits and (possibly sharp) bounds.
The derivation of proofs of the observed results is left for future work.  
Nevertheless, there are some observations which can be stated without detailed mathematics.
\begin{enumerate}
\item
  Possibly the most important finding is that, for $k>1$, it is much better to work with
  $\kappa = k(k+1)/2$ and the mean $\mu=\kappa\lambda$ instead of $k$ and $\lambda$ themselves.
Better results are obtained thereby.
It might in fact be better to characterize the Poisson distribution of order $k$ by its mean, not by $\lambda$.
\item
  This note clarifies the relation of ``$n$'' to the mean, median and mode for $k>1$.
  In particular, if the mean equals $n$, i.e.~$\lambda = n/\kappa$, this note conjectures {\em exact} expressions
  for the median (for $n\ge\kappa$) and the mode (for $n\ge2\kappa$).
\item
The number $8$ appears frequently in the asymptotic expressions for both the median and mode of the Poisson distribution of order $k$.
For example $(k+4)/8$ for the median and $(3k+5)/8$ for the mode 
and the exponent for $(\kappa\lambda)_0$ ($=9/8$, see Fig.~\ref{fig:lamfirst_mode}).
The reason for this is unknown.
\item
  The quality of the numerical fits for the median is better than the quality of the fits for the mode. 
In general, it was observed that the numerical errors and stability of the calculations were better for the median than the mode.
Possibly the reason is that the mode is determined by the comparison of the heights of individual histogram bins in the probability mass function,
whereas the median is computed using the {\em sum} of multiple bins,
and the summation over histogram bins smooths out the numerical fluctuations in the values of the individual bins.
\end{enumerate}

\section*{\label{sec:ack}Acknowledgements}
The author is indebted to Professor A.~N.~Philippou for bringing the Poisson distribution of order $k$ to his attention,
and for stimulating his interest in the numerical analysis of the properties of the distribution.

\newpage
\begin{table}[htb]
\centering
\begin{tabular}[width=0.75\textwidth]{|l|lllllll|}
  \hline
  $n$ & $\lambda^0(=1)$ & $\lambda$ & $\lambda^2$ & $\lambda^3$ & $\lambda^4$ & $\lambda^5$ & $\lambda^6$ \\
  \hline
  0 & $(0,0)$ &&&&&& \\
  \hline
  1 & & $(1,0)$ &&&&& \\
  \hline
  2 & & $(0,1)$ & $(2,0)$ &&&& \\
  \hline
  3 & & & $(1,1)$ & $(3,0)$ &&& \\
  \hline
  4 & & & $(0,2)$ & $(2,1)$ & $(4,0)$ && \\
  \hline
  5 & & & & $(1,2)$ & $(3,1)$ & $(5,0)$ & \\
  \hline
  6 & & & & $(0,3)$ & $(2,2)$ & $(4,1)$ & $(6,0)$ \\
  \hline
  $\vdots$ & $\cdots$ &&&&&& \\
  \hline
\end{tabular}
\caption{\label{tb:tablehk}
Tabulation of tuples $(n_1,n_2)$ for the Poisson distribution of order $2$ and parameter $\lambda$.}
\end{table}

\newpage
\begin{figure}[!htb]
\centering
\includegraphics[width=0.75\textwidth]{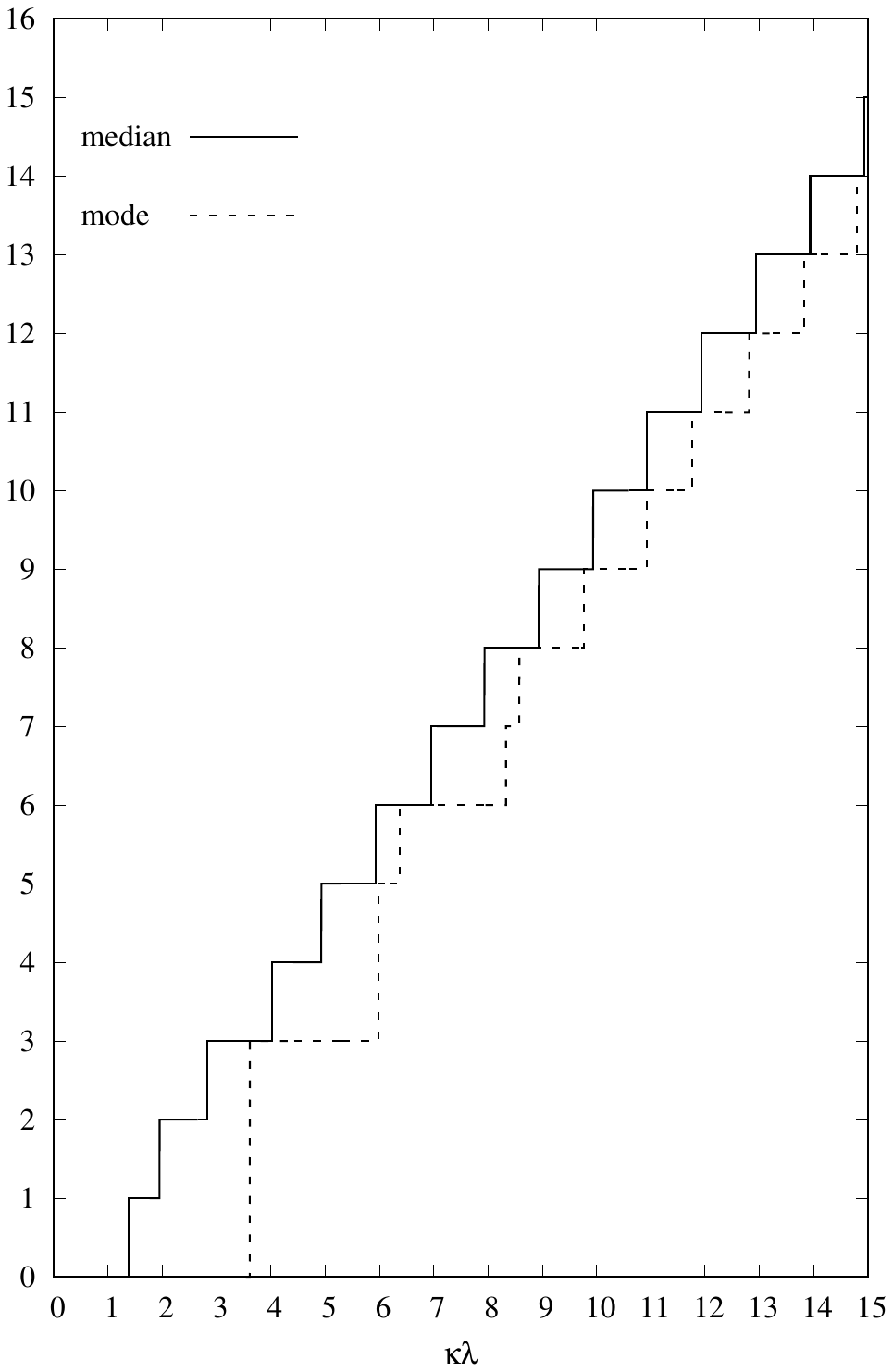}
\caption{\small
\label{fig:median_mode_k3}
Plot of the median (solid) and mode (dashed) as a function of $\kappa\lambda$ for the Poisson distribution of order $3$.}
\end{figure}

\newpage
\begin{figure}[!htb]
\centering
\includegraphics[width=0.75\textwidth]{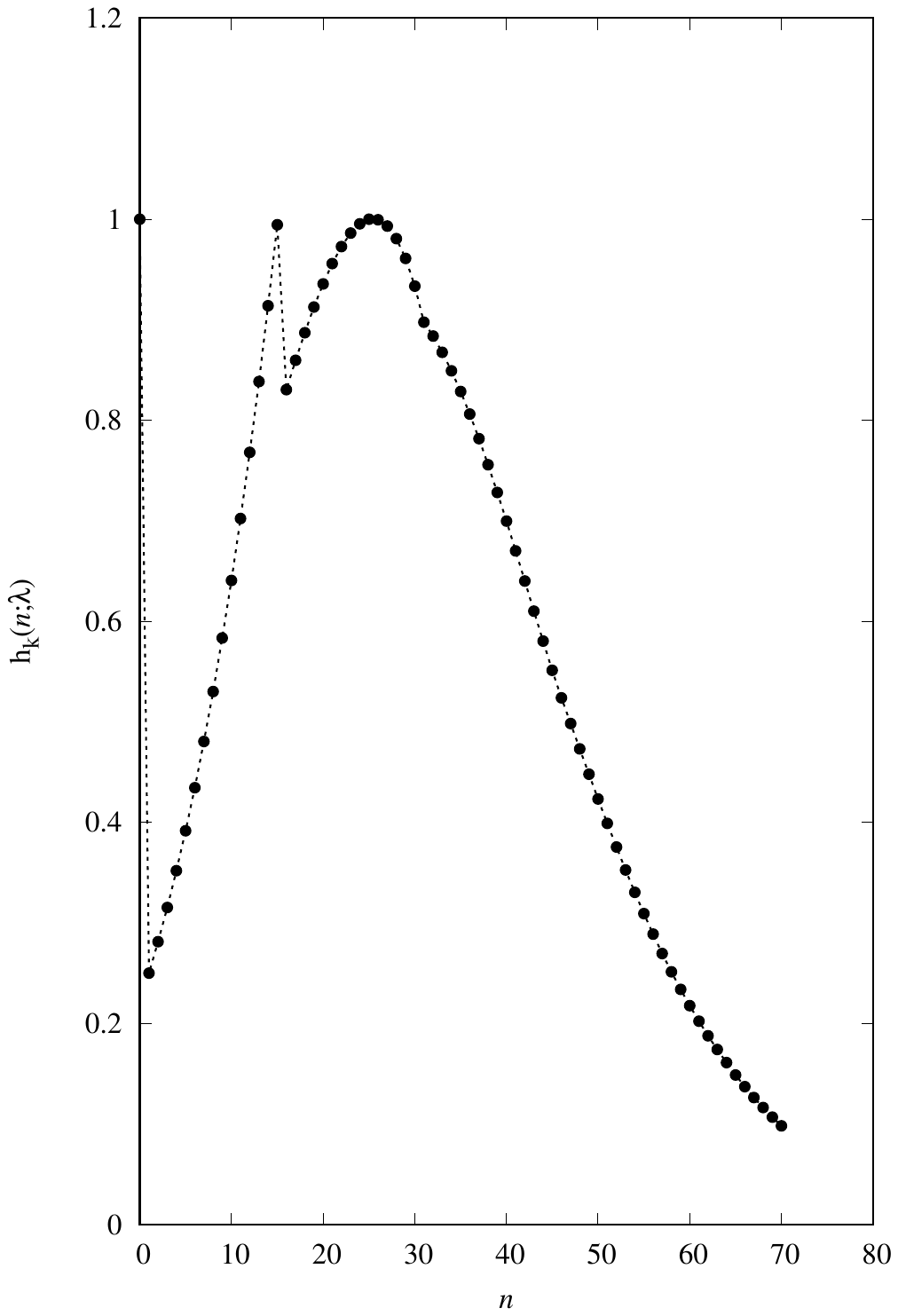}
\caption{\small
\label{fig:histk15_dbl_mode}
Histogram plot of $h_k(n;\lambda)$ for the Poisson distribution of order $15$, exhibiting the first double mode (at $n=0$ and $25$).}
\end{figure}

\newpage
\begin{figure}[!htb]
\centering
\includegraphics[width=0.75\textwidth]{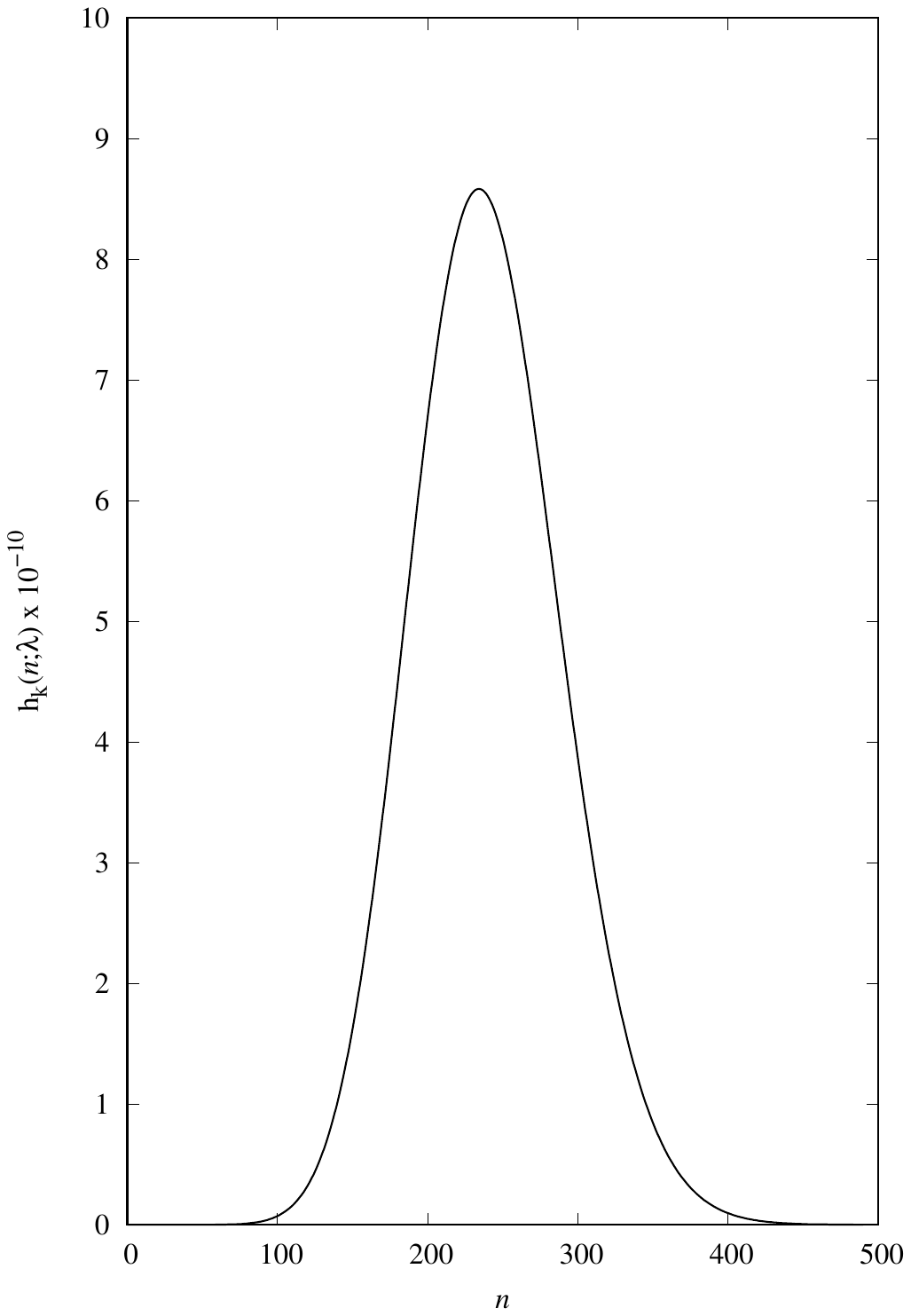}
\caption{\small
\label{fig:histk15_lam2}
Histogram plot of $h_k(n;\lambda)$ for the Poisson distribution of order $15$ for $\lambda=2$.}
\end{figure}

\newpage
\begin{figure}[!htb]
\centering
\includegraphics[width=0.75\textwidth]{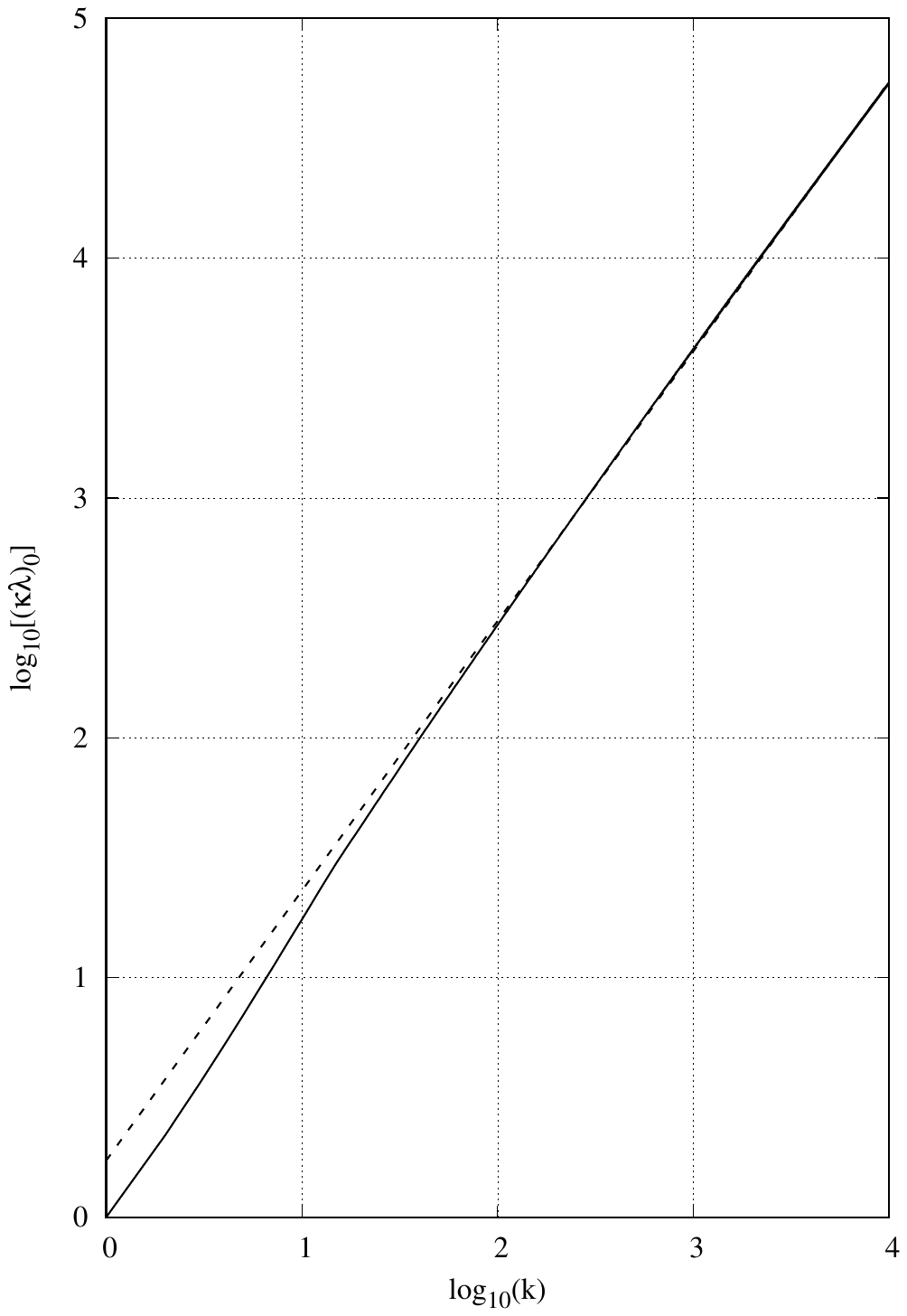}
\caption{\small
\label{fig:lamfirst_mode}
Logarithmic plot of the value of $(\kappa\lambda)_0$ for the Poisson distribution of order $k$.
The dashed line is a power law fit $\textrm{const}\times k^{1.125}$.}
\end{figure}

\end{document}